\documentclass[a4paper,11pt]{amsart} 
\usepackage{amssymb, enumerate, graphicx, enumitem, color, comment}
\usepackage[dvipsnames,svgnames,table]{xcolor}
\usepackage{hyperref}
\hypersetup{
  pdfauthor={Paul Bastide, Jędrzej Hodor, Hoang La, William T. Trotter},
  pdftitle={Cube width},
    colorlinks,
    linkcolor={RoyalBlue},
    citecolor={RubineRed},
    urlcolor={blue!80!black}
}
\usepackage[capitalise, noabbrev]{cleveref}
\usepackage{mathtools}
\usepackage[foot]{amsaddr}
\newcommand{\q}[1]{``#1''}
\newcommand{\defin}[1]{\emph{\textcolor{ForestGreen}{#1}}}

\newenvironment{proofclaim}[1][]
	{\par\noindent {\it Proof of the claim}. }{ \hfill$\lozenge$\par\vspace{11pt}}
\newenvironment{proofsubclaim}[1][]
	{\par\noindent {\it Proof of the subclaim}. }{ \hfill$\circ$\par\vspace{11pt}}

\usepackage[norule]{footmisc}


\usepackage[T1]{fontenc}
\theoremstyle{plain} 
\newtheorem{theorem}{Theorem}
\newtheorem{corollary}[theorem]{Corollary}

\newtheorem{lemma}[theorem]{Lemma} 

\newtheorem{proposition}[theorem]{Proposition}
 
\newtheorem{example}[theorem]{Example}
\theoremstyle{remark}

\newtheorem*{claim*}{Claim} 
\newtheorem*{subclaim*}{Subclaim} 
\bibstyle{plain}

\DeclarePairedDelimiter\inrep{(}{)}

\newcommand{\bbA}{\mathbb{A}}
\newcommand{\bbB}{\mathbb{B}}

\newcommand{\cgC}{\mathcal{C}}

\newcommand{\cgF}{\mathcal{F}}

\newcommand{\Oh}{\mathcal{O}}
 
\newcommand{\cgR}{\mathcal{R}}
\newcommand{\cgS}{\mathcal{S}}
\newcommand{\cgT}{\mathcal{T}}

\newcommand{\ch}{\operatorname{ch}}
\newcommand{\cw}{\operatorname{cw}}
\newcommand{\iir}{\operatorname{iir}}

\newcommand{\MCW}{\mathbb{MCW}} 
\newcommand{\MTD}{\mathbb{MTD}} 
\newcommand{\MIIR}{\mathbb{MIIR}} 
\newcommand{\NMIIR}{\mathbb{NMIIR}} 

\let\le\leqslant
\let\ge\geqslant
\let\leq\leqslant
\let\geq\geqslant

\let\subset\subseteq
\let\subsetneq\varsubsetneq

\let\epsilon\varepsilon

\begin{document}

\title[CUBE HEIGHT AND CUBE WITDH]%
{Cube Height, Cube Width and Related\\ 
Extremal Problems for Posets}

\author[Bastide]{Paul Bastide}

\author[Hodor]{J\k{e}drzej Hodor}

\author[La]{Hoang La}

\author[Trotter]{William T. Trotter}

\thanks{(P.~Bastide) \textsc{Mathematical Institute, University of Oxford, United Kingdom.} 
(J.~Hodor) \textsc{Theoretical Computer Science Department, Faculty of Mathematics and Computer Science and  Doctoral School of Exact and Natural Sciences, Jagiellonian University, Krak\'ow, Poland.}
(H.~La) \textsc{LISN, Universit\'e Paris-Saclay, CNRS, Gif-sur-Yvette, France.}
(W.~T.~Trotter) \textsc{School of Mathematics, Georgia Institute of Mathematics, Atlanta, Georgia 30332, U.S.A}}

\thanks{
\textit{E-mail addresses:} \href{mailto:paul.bastide@ens-rennes.fr}{paul.bastide@ens-rennes.fr},
\href{mailto:jedrzej.hodor@gmail.com}{jedrzej.hodor@gmail.com},
\href{mailto:hoang.la.research@gmail.com}{hoang.la.research@gmail.com},
\href{mailto:wtt.math@gmail.com}{wtt.math@gmail.com}}

\thanks{P. Bastide is supported by ERC Advanced Grant
883810, and J.~Hodor is supported by the National Science Center of Poland 
under grant UMO-2022/47/B/ST6/02837 within the OPUS 24 program.}

\thanks{\noindent\rule{\textwidth}{0.5pt}}

\thanks{\textbf{Acknowledgments:}
The research was conducted during the 2024 Order \& Geometry Workshop in Wittenberg.
We are grateful to the organizers and participants for creating a friendly and stimulating environment.
Some of the results are part of the PhD thesis of the first author~\cite{paulthesis} defended in June 2025.
Shortly before this paper was made public, another independent manuscript in which~\Cref{thm:cw-bound} is proved had been submitted to arxiv (Flídr, Ivan, and Jaffe~\cite{FIJ25}).
}

\thanks{\noindent\rule{\textwidth}{0.5pt}}

\begin{abstract} 
  Given a poset $P$, a family $\cgS=\inrep{S_x:x\in P}$ of sets indexed by the 
  elements of $P$ is called an inclusion representation of $P$ if 
  $x\le y$ in $P$ if and only if $S_x\subseteq S_y$. The cube height of a poset
  is the least non-negative integer $h$ such that $P$ has an inclusion representation
  for which every set has size at most~$h$.  In turn, the cube width of $P$
  is the least non-negative integer $w$ for which there is 
  an inclusion representation $\cgS$ of $P$ such that $|\bigcup\cgS|=w$ and every 
  set in $\cgS$ has size at most the cube height of $P$.  In this paper, we show that
  the cube width of a poset never exceeds the size of its ground set, and we
  characterize those posets for which this inequality is tight.  
  Our research prompted us to investigate related extremal problems for
  posets and inclusion representations.  Accordingly, the results for cube width
  are obtained as extensions of more comprehensive results that we believe to
  be of independent interest.
\end{abstract}

\maketitle

\thispagestyle{empty}


\section{Introduction}
\medskip

We consider only finite posets with non-empty ground sets.\footnote{We allow subposets to be empty.}
Given a poset $P$, a family $\cgS=\inrep{S_x:x\in P}$ of sets is 
called an \defin{inclusion representation} of $P$ if for all $x,y\in P$, we 
have $x\le y$ in $P$ if and only if $S_x\subseteq S_y$. 
Every poset has an inclusion representation.
, which we call the \defin{canonical inclusion representation} of a poset. 
For a poset $P$ and an element $x\in P$, let \defin{$D_P[x]$} denote the 
\defin{closed down set} of $x$ in $P$, i.e., the set of all $u\in P$ such 
that $u\leq x$ in $P$. 
One can easily verify that $\inrep{D_P[x]:x\in P}$ is an inclusion representation of $P$.
We call this particular representation the \defin{canonical inclusion representation} of $P$.


When $\cgS=\inrep{S_x:x\in P}$ is an inclusion representation of a poset $P$,
we refer to $\bigcup\cgS$ as the \defin{ground set} of $\cgS$.  
The \defin{cube height} of a poset $P$, denoted by \defin{$\ch(P)$}, is 
the least non-negative integer $h$ such that $P$ has an inclusion representation 
$\inrep{S_x:x\in P}$ with $|S_x|\leq h$ for every $x\in P$.
The \defin{cube width} of a poset $P$, denoted by \defin{$\cw(P)$}, is the least non-negative integer 
$w$ for which there is an inclusion representation $\cgS$ of $P$ such
that $|\bigcup\cgS|= w$, and $|S_x|\leq\ch(P)$ for every $x\in P$.
See~\Cref{fig:cw-fig1}.


\begin{figure}[tp]
   \begin{center}
     \includegraphics{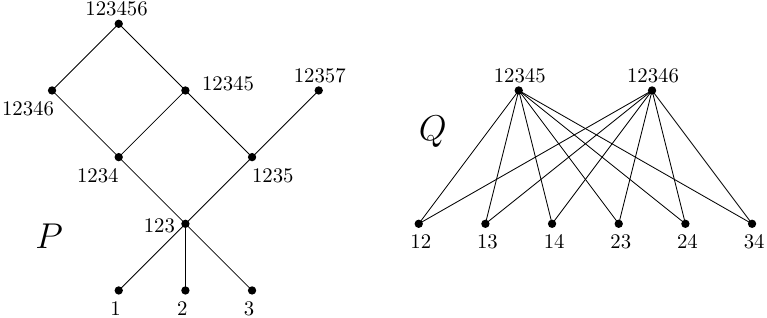}

   \end{center}
   \caption{Inclusion representations of posets $P$ and $Q$ are shown.   Sets
   are shown without braces and without commas.  It is straightforward to verify
   that $\ch(P)=6$ and $\cw(P)=7<|P|=10$.  Also,
   $\ch(Q)=5$ and $\cw(Q)=6<|Q|=8$.}
   \label{fig:cw-fig1}
\end{figure}


For a poset $P$, let $|P|$ denote the cardinality of its ground set.
The first author, Groenland, Ivan, and Johnston~\cite{Bastide24} investigated
a poset parameter called ``induced saturation number'', and their research
led them to formulate the definitions of cube height and cube width.  
They showed that $\cw(P)\leq |P|^2\slash 4 + 2$ for 
every poset $P$, and they conjectured that $\cw(P)\leq|P|$ for every poset $P$.  
In this paper, we resolve their conjecture in the affirmative by proving 
the following theorem.  

\begin{theorem}\label{thm:cw-bound}
  For every poset $P$, $\cw(P)\leq|P|$.
\end{theorem}

\Cref{thm:cw-bound} directly improves the best-known general bound on the induced saturation number of posets~\cite{Bastide24}.
When $n$ is a non-negative integer, we use the abbreviation $[n]$ for the set of the $n$ least positive integers.
Given a poset $P$ and a positive integer $n$, we say that a family $\cgF$ of subsets of $[n]$ is \defin{$P$-saturated} if $\cgF$ does not contain an induced copy of $P$ ($[n]$ is treated as a poset ordered by inclusion), but adding any other set to $\cgF$ creates an induced copy of $P$.
For a poset $P$ and a positive integer $n$, the $n$th \defin{induced saturation number} of $P$, denoted by \defin{$\mathrm{sat}^*(n,P)$}, is the size of the smallest $P$-saturated family of subsets of $[n]$. 
It was shown that $\mathrm{sat}^*(n,P) = \Oh(n^{\cw(P)-1})$~\cite{Bastide24}, and therefore,~\Cref{thm:cw-bound} yields the following.

\begin{corollary}
    For every poset $P$, and every $n \in \mathbb{N}$, $\operatorname{sat}^*(n,P) = O(n^{|P|-1})$.
\end{corollary}

The key idea of our proof of~\Cref{thm:cw-bound} is to design the correct induction statement.
In particular, we prove a stronger statement regarding inclusion representations of posets that immediately implies~\Cref{thm:cw-bound}.
Note that we give a short, self-contained, and elementary proof of this stronger result.
In order to state it, we need several additional definitions.


\subsection{Irreducible inclusion representations}
Let $P$ be a poset, and let $\cgS = \inrep{S_x:x\in P}$ and 
$\cgS'=\inrep{S_x':x\in P}$ be inclusion representations of $P$.
We say that $\cgS$ and $\cgS'$ are \defin{isomorphic} if there is a bijection 
$f\colon\bigcup\cgS\to\bigcup\cgS'$ such that for every $x\in P$ and every 
$a\in\bigcup\cgS$, $a\in S_x$ if and only if $f(a)\in S'_x$. 
We say that $\cgS$ is a \defin{reduction} of $\cgS'$ if 
$|\bigcup\cgS|\le|\bigcup\cgS'|$ and $|S_x|\le|S'_x|$ for every $x\in P$.
With this definition, an inclusion representation is a reduction of itself.
We say that $\cgS$ is \defin{equivalent} to $\cgS'$ if $\cgS$ is a reduction of 
$\cgS'$ and $\cgS'$ is a reduction of $\cgS$.

Clearly, if $\cgS$ and $\cgS'$ are inclusion representations of a poset $P$, and
they are isomorphic, they are equivalent.  On the other hand, 
an equivalence class can consist of arbitrarily many different isomorphism classes.
To see this, let $s$ be an integer with $s\ge3$, and set $t=\binom{2s}{s}$.
Then let $P$ be a poset of height~$2$ such that (1)~$|P|=t+2$; 
(2)~$P$ has $t$ minimal elements; (3)~the remaining two elements of $P$ 
are incomparable and each covers $s+1$ minimal elements;
(4)~no minimal element has two upper covers.
Let $\cgT$ consist of all $s$ element subsets of $[2s]$, and set
$S_1=\{1,\dots,s+1\}$.  Then, to construct a family of sets that is
an inclusion representation of $P$, we simply add to
$\cgT\cup\{S_1\}$ one additional set of the form $\{i,i+1,\dots,i+s\}$, 
where $3\le i\le s$.  
The assignments of sets in the representation to elements of posets are clear.
Distinct choices for $i$ yield inclusion representations
that are equivalent but not isomorphic.

When $\cgS$ and $\cgS'$ are inclusion
representations of $P$, we say that $\cgS$ is a \defin{strict reduction} of 
$\cgS'$ if $\cgS$ is a reduction of $\cgS'$ but they are not equivalent. 
An inclusion representation that has no strict reduction is said to be \defin{irreducible}. 
We only consider finite poset, therefore, given any inclusion representation $\cgS$
of a poset $P$, it is clear that either $\cgS$ is irreducible or there is an irreducible inclusion
representation $\cgS'$ of $P$ such that
$\cgS'$ is a strict reduction of $\cgS$.  Both inclusion representations shown in~\Cref{fig:cw-fig1} are irreducible.
Note that the poset $Q$ in~\Cref{fig:cw-fig1} has (at least) two irreducible inclusion representations with ground sets of different size.
Indeed, if $\cgS$ is the representation given in the figure and $\cgC$ is the canonical inclusion representation, then, $|\bigcup \cgS| = 6$ and $|\bigcup \cgC| = 8$.
Also, both $\cgS$ and $\cgC$ are irreducible. 


A natural extremal problem that arises is to find the maximum size of the ground set of an irreducible inclusion representation of a given poset.
Thus, for a poset $P$, we define \defin{$\iir(P)$} as the maximum non-negative integer $w$ such that there is an irreducible inclusion representation of $P$ with the ground set of size $w$.
Note that, for every poset $P$, we have $\cw(P) \leq \iir(P)$.
Indeed, for a poset $P$, take an inclusion representation $\cgS$ of $P$ witnessing $\ch(P)$ and then set $\cgS'$ to be an irreducible inclusion representation of $P$ such that $\cgS'$ is a reduction of $\cgS$.
Therefore, the following result implies~\cref{thm:cw-bound}.

\begin{theorem}\label{thm:iir-bound}
  For every poset $P$, $\iir(P)\leq|P|$.
\end{theorem}

The inequality in~\Cref{thm:iir-bound} can be strict.
For instance, as discusses before, the poset $Q$ in~\Cref{fig:cw-fig1} satisfies $\iir(Q) \geq 8 = |Q|$ (witnessed by the canonical inclusion representation), hence, $6 = \cw(Q)<\iir(Q)=|Q|=8$.

\subsection{Other related notions}
The definition of cube width may seem slightly convoluted as it comes straight 
from the application discussed in~\cite{Bastide24}.  However, there is 
a much simpler parameter that is relevant to our study.  For a poset $P$, 
the \defin{$2$-dimension} of $P$, 
denoted \defin{$\dim_2(P)$}, is the least non-negative integer $w$ such that there is an 
inclusion representation $\cgS=\inrep{S_x:x\in P}$ of $P$ with $|\bigcup\cgS|= w$.
Note that there is no restriction on the sizes of sets in $\cgS$.  The concept of 
$2$-dimension is a generalization of the celebrated notion of dimension 
introduced in 1941 by Dushnik and Miller~\cite{DM41}.  Also, $2$-dimension 
generalizes to $k$-dimension for any integer $k \geq 2$~\cite{Trotter76}.
The concept of $2$-dimension was first studied by Nov\'ak~\cite{Novak63} in 
1963 and later by others~\cite{Trotter75, Trotter75-2, BaymWest12, LewisSouza21}.

Canonical inclusion representations witness that $\dim_2(P)\leq |P|$ 
for every poset $P$.  Also, for every poset $P$, we have
\[
  \ch(P) \leq \dim_2(P) \leq \cw(P) \leq \iir(P).
\]
As noted just above, the inequality $\cw(P)\le\iir(P)$ can be strict.
In order to see that the inequalities $\ch(P)\leq\dim_2(P)$ and
$\dim_2(P)\le\cw(P)$ can be strict, we simply consider a large enough antichain $P$.
Then, $\ch(P) = 1$ and so $\cw(P) = |P|$.
On the other hand, $P$ can be realized as larger subsets.
Namely, let $n$ and $t$ be positive integers such that $\binom{n}{t} \geq |P|$ and realize $P$ as $t$-subsets of $[n]$.
In particular, $\dim_2(P) \leq n$, whereas $n$ can be much smaller than $|P|$.
More precisely, Sperner's theorem implies that
$\dim_2(P)$ is the least $s$ such that $\binom{s}{\lfloor s/2\rfloor}\ge |P|$, which is of order $\Theta(\log |P|)$.
Finally, note that for an antichain $P$, we have $\dim_2(P) < |P|$ whenever $|P| \geq 5$.

Several interesting properties of $2$-dimension are known~\cite{Trotter75}.
First, $2$-dimension is \defin{monotonic}, that is, if $Q$ is a subposet of $P$, 
then $\dim_2(Q)\le\dim_2(P)$.  Second, abusing terminology slightly, we say 
that $2$-dimension is \defin{continuous}, i.e., small changes in the poset 
can only produce small changes in its $2$-dimension. 
Specifically, if $x$ is an element in a non-trivial poset $P$, 
and $Q=P-\{x\}$, then $|\dim_2(P)-\dim_2(Q)|\le 2$. 
There are simple examples to show that this inequality can be tight.

Although cube height is easily seen to be monotonic, the
next example shows that cube height is not continuous.  The example
also shows that cube width and the maximum size of a ground set of an irreducible inclusion representation ($\iir(\cdot)$) are neither monotonic nor continuous.

\begin{example}
  Let $t$ be a positive integer with $t\ge3$, and set $s=\binom{2t+1}{t}$.  
  Let $Q$ be an $s$-element antichain, and let $P$ be the poset obtained
  from $Q$ by adding a unique maximal element.  
  Then the following statements hold: 
  \begin{enumerate}
    \item $\ch(P) = \dim_2(P) = \cw(P) = \iir(P)=2t+1$;
    \item $\ch(Q)=1$, $\dim_2(Q) = 2t+1$, and $\cw(Q)=\iir(P)=s$.
  \end{enumerate}
\end{example}

\subsection{Posets with \q{wide} inclusion representations}
By~\Cref{thm:iir-bound}, we have $\ch(P) \leq \dim_2(P) \leq \cw(P) \leq \iir(P) \leq |P|$.
Once we have this chain of inequalities, it is natural to consider the associated characterization problems.
Given a poset $P$, can we detect with a polynomial time algorithm whether or not $p(P) = |P|$ for each parameter $p \in \{\ch,\dim_2,\cw,\iir\}$.
For the cube height, one can immediately see that there are no posets $P$ with $\ch(P) = |P|$.
For $2$-dimension, the characterization was already given in~\cite{Trotter75}.
However, a cleaner and more elegant proof emerges from our more comprehensive results with essentially no extra effort required.
We prove the following result.

\begin{theorem}\label{thm:poly}
    For each parameter $p \in \{\dim_2,\cw,\iir\}$ there is a polynomial-time algorithm, which for a poset $P$ decides if $p(P) = |P|$.
\end{theorem}

In the case of $p \in \{\dim_2,\cw\}$, we give a very precise description  of posets $P$ with $p(P) = |P|$.
On the other hand, the algorithm is less direct in the case of $p = \iir$.
See Theorems~\ref{thm:char:iir},~\ref{thm:char:dim2}, and~\ref{thm:char:cw} for the detailed statements.

\subsection{Outline of the paper}
The remainder of the paper is organized as follows.  
We start by fixing some more notation in~\Cref{sec:notation}.
In~\Cref{sec:main-proof}, we give a short and self-contained proof of~\Cref{thm:iir-bound} followed by several corollaries in~\Cref{sec:implications}.
The rest of the paper is devoted to proving~\Cref{thm:poly}.
First, in~\Cref{sec:combining-splitting} we provide some additional material on inclusion representations.
In~\Cref{sec:char-thms-iir}, we characterize posets $P$ with $\iir(P) = |P|$ (in other words, we prove~\Cref{thm:poly} for $p = \iir$).
In~\Cref{sec:dimcw}, we characterize posets $P$ with $\dim_2(P) = |P|$ and we characterize posets $P$ with $\cw(P) = |P|$ (in other words, we prove~\Cref{thm:poly} for $p \in \{\dim_2,\cw\}$).
Finally, in~\Cref{sec:conclusion}, we suggest further research directions by revisiting the topic of structural properties of those posets $P$ for which $\iir(P)=|P|$.

\subsection{Basic notation and conventions} \label{sec:notation}
All the poset parameters we discuss in this paper are the 
same whenever $P$ and $Q$ are isomorphic posets.  Accordingly, we abuse 
notation slightly and say $P=Q$ when $P$ and $Q$ are isomorphic.  Also, we say
$P$ contains $Q$ when $Q$ is isomorphic to a subposet of $P$.
We treat subset of elements of a poset as posets with induced order relation.

Let $P$ be a poset.
We write \defin{$x\in P$} when $x$ is a member of the
ground set of $P$.
When $x$ and $y$ are distinct incomparable elements of 
$P$, we will write \defin{$x\parallel y$} in $P$.  When
$x\in P$, we let \defin{$D_P(x)$} consist of all elements $y\in P$ such that 
$y< x$ in $P$.  
Moreover, \defin{$D_P[x]$} is defined to be $D_P(x)\cup\{x\}$.  The sets
\defin{$U_P(x)$} and \defin{$U_P[x]$} are defined dually. 
A subposet $Q$ of a poset $P$ is called a
\defin{down set} (resp.\ \defin{up set}) in $P$ if for all $u,v \in P$ with $u < v$ in $P$ $u\in Q$ implies $v\in Q$ (resp.\ $v\in Q$ implies $u\in Q$).

We use standard terminology regarding covers, i.e., we say $y$ \defin{covers}
$x$ in $P$ when $x<y$ in $P$ 
and there is no element $z$ of $P$ with
$x<z<y$ in $P$.  Every element $y$ of $P$ that is not a minimal element covers
at least one element of $P$.  Later in the paper, a key detail will hinge
on whether $y$ covers at least two elements of $P$.

\medskip
\section{Proof of Theorem~\ref{thm:iir-bound}} \label{sec:main-proof}
\medskip

The key part of the proof is the following technical lemma.

\begin{lemma}\label{lem:key-step}
    Let $P$ be a poset, let $\cgS = \inrep{S_x : x \in P}$ be an inclusion representation of $P$, and let $y \in P$.
    Let $Q = P - D_P[y]$, and let $Q'$ be a poset with the ground set $\{S_x-S_y : x \in Q\}$ ordered by inclusion.
    Let $\varepsilon \in \{0,1\}$ be the number of unique minimal elements of $Q'$.
    Then, there exists an inclusion representation $\cgS' = \inrep{S_x' : x \in P}$ of $P$ with $|S_x'| \leq |S_x|$ for every $x \in P$ and
        \[\left|\bigcup\cgS'\right| \leq \iir(Q') + \varepsilon + |S_y|.\]
\end{lemma}
\begin{proof} 
    Let $A = \emptyset$ when $Q'$ has no unique minimal element, and $A$ be the unique minimal element of $Q'$ otherwise.
    Note that in the latter case $A\neq \emptyset$ since otherwise $q \leq y$ in $P$ where $q \in P - D_P[y]$ is such that $S_q-S_y = A$ -- this is a contradiction.
    Clearly, $\cgT' = \inrep{\alpha - A : \alpha \in Q'}$ is an inclusion representation of $Q'$.
    Let $\cgR' = \inrep{R_\alpha' : \alpha \in Q'}$ be an irreducible inclusion representation of $Q'$ that is a reduction of $\cgT'$.
    In particular, $|\bigcup \cgR'|= \iir(Q')$.
    Assume that the ground sets of $\cgS$ and $\cgR'$ are disjoint.
    For every $x \in D_P[y]$, let $S_x' = S_x$. 
    Choose $A' \subset A$ arbitrarily so that $\varepsilon = |A'|$ and $A' \cap S_y = \emptyset$.
    For every $x \in P - D_P[y]$, let $S_x' = R_\alpha' \cup (S_x \cap S_y) \cup A'$ where $\alpha \in Q'$ is such that $\alpha = S_x-S_y$.
    We claim that $\cgS' = \inrep{S_x' : x \in P}$ is an inclusion representation of $P$ and $|S_x'| \leq |S_x|$ for every $x \in P$.

    Let $x,z \in P$.
    If $x \leq z$ in $P$, then $S_x' \subset S_z'$ by definition.
    Thus, assume $x \not< z$ in $P$.
    If $x,z \in D_P[y]$, then $S'_x \not\subset S'_z$ as $\cgS$ is an inclusion representation of $P$.
    If $x,z \notin D_P[y]$, then either $(S_x \cap S_y) \not\subset (S_z \cap S_y)$ or $(S_x - S_y) \not\subset (S_z - S_y)$.
    In the former case clearly $S_x' \not\subset S_z'$ and in the latter case we have $R_\alpha' \not\subset R_\beta'$, where $\alpha = S_x-S_y$ and $\beta = S_z-S_y$, and so, $S_x' \not\subset S_z'$.
    Next, suppose that $x \in D_P[y]$ and $z \notin D_P[y]$.
    If $S_x' \subset S_z'$, then $S_x \subset S_z \cap S_y \subset S_z$, which is not possible.
    Finally, assume $x \notin D_P[y]$ and $z \in D_P[y]$.
    If $Q'$ has no unique minimal element, then $R_\alpha \neq \emptyset$ for every $\alpha \in Q'$, and so, $S_x' \not\subset S_z'$.
    Otherwise, $A' \subset S_x' - S_y \subset  S_x' - S_z'$, which implies $S_x' \not\subset S_z'$.
    The above case analysis yields that $\cgS'$ is indeed an inclusion representation of $P$.
    Now, we argue that $\cgS'$ is a reduction of $\cgS$.
    For every $x \in D_P[y]$, we have $|S'_x| = |S_x|$, and for every $x \in P - D_P[y]$, we have 
    \[|S'_x| = |R_\alpha'| + |S_x \cap S_y| + |A'| \leq |\alpha - A| + |S_x \cap S_y| + |A'| = |S_x-S_y| - |A| + |S_x \cap S_y| + |A'| \leq |S_x|.\]
    Finally,
    \[\left|\bigcup\cgS'\right| \leq \left|\bigcup \cgR'\right| + |A'| + \left|S_y\right| = \iir(Q') + \varepsilon + |S_y|. \qedhere\]
\end{proof}

\begin{proof}[Proof of~\Cref{thm:iir-bound}]
    The proof is by induction on the number of elements of $P$.
    If $P$ is trivial, then the statement is clear.
    Suppose that $P$ is non-trivial and let $\cgS = \inrep{S_x : x \in P}$ be an irreducible inclusion representation of $P$.
    Additionally, suppose to the contrary that $|\bigcup \cgS| > |P|$.
    If $|D_P[x]| \leq |S_x|$ for every $x \in P$, then the canonical inclusion representation of $P$ is a strict reduction of $\cgS$, which contradicts the irreducibility of $\cgS$.

    Therefore, we can assume that there is $y \in P$ with $|D_P[y]| > |S_y|$.
    Observe that $y$ is not a unique maximal element in $P$ as otherwise $|S_y| = |\bigcup \cgS| > |P| = |D_P[y]|$.
    In particular, $Q = P - D_P[y]$ is non-empty.
    Consider a poset $Q'$ with the ground set $\{S_x - S_y : x \in Q\}$ equipped with the inclusion relation and let $\varepsilon \in \{0,1\}$ be the number of unique minimal elements of $Q'$.
    Note that by induction $\iir(Q') \leq |Q'|$.
    By~\Cref{lem:key-step}, there is an inclusion representation $\cgS' = \inrep{S_x' : x \in P}$ of $P$ with $|S_x'| \leq |S_x|$ for every $x \in P$ and
        \[\left|\bigcup\cgS'\right| \leq \iir(Q') + \varepsilon + |S_y| \leq |Q'| + 1 + (|D_P[y]-1) \leq |P| < \left|\bigcup\cgS\right|.\]
    This shows that $\cgS'$ is a strict reduction of $\cgS$, which is a contradiction that completes the proof.
\end{proof}

\subsection{Some implications}\label{sec:implications}

The bound can be slightly improved when a poset has a unique minimal element.

\begin{corollary} \label{cor:irr-unique-minimal-element}
    Every irreducible inclusion representation of every poset $P$ with a unique minimal element uses at most $|P| - 1$ elements.
\end{corollary}
\begin{proof}
    We proceed by induction on $|P|$.
    In a one-element poset, we can represent the unique element as the empty set and such a representation is a reduction of any other.
    Let $P$ be a poset with a unique minimal element $y$ and $\cgS = \inrep{S_x : x \in P}$ be an irreducible inclusion representation of $P$.
    Let $Q = P - \{y\}$ and let $S_x' = S_x - S_y$ for every $x \in Q$.
    Clearly, $\cgS' = \inrep{S_x' : x\in Q}$ is an inclusion representation of $Q$.
    Let $\cgS'' = \inrep{S_x'' : x \in Q}$ be an irreducible reduction of $\cgS'$.
    By~\Cref{thm:iir-bound}, $|\bigcup \cgS''| \leq |Q|$.
    Moreover, if $Q$ has a unique minimal element, by induction, $|\bigcup \cgS''| \leq |Q|-1$.
    First, assume that $Q$ has no unique minimal element.
    Then, every set in $\cgS''$ is non-empty and by setting $T_x = S_x''$ for every $x \in Q$ and $T_y = \emptyset$, we obtain an inclusion representation $\inrep{T_x : x \in P}$ of $P$ that is a reduction of $\cgS$ and uses at most $|P|-1$ elements.
    Finally, when $Q$ has a unique minimal element, we set $T_y = \emptyset$, we pick $\gamma$ to be an element not used in any inclusion representation considered before, and we set $T_x = S_x'' \cup \{\gamma\}$ for every $x \in P - \{y\}$.
    Again, $\inrep{T_x : x \in P}$ is an inclusion representation of $P$,  a reduction of $\cgS$, and uses at most $|P|-1$ elements.
\end{proof}

The next two corollaries aim to give more insight into posets with $\iir(P) = |P|$.
Note that we do not obtain directly a polynomial time detection algorithm yet.

\begin{corollary}\label{cor:irr-IR-downsets}
    Let $P$ be a poset and let $\cgS = \inrep{S_x : x \in P}$ be an irreducible inclusion representation of $P$.
    If $|\bigcup \cgS| = |P|$, then $|S_x| = |D_P[x]|$ for every $x \in P$.
\end{corollary}
\begin{proof}
    Assume that $|\bigcup \cgS| = |P|$.
    Note that it suffices to show that $|D_P[x]| \leq |S_x|$.
    Indeed, if now any of the inequalities is strict, then the canonical inclusion representation of $P$ is a strict reduction of $\cgS$, which is a contradiction.
    Suppose to the contrary that there is $y \in P$ with $|D_P[y]| > |S_y|$.
    Observe that $y$ is not a unique maximal element in $P$ as otherwise $|S_y| = |\bigcup \cgS| = |P| = |D_P[y]|$, hence $Q = P - D_P[y]$ is non-empty.
    Consider a poset $Q'$ with the ground set $\{S_x - S_y : x \in Q\}$ equipped with the inclusion relation and let $\varepsilon \in \{0,1\}$ be the number of unique minimal elements of $Q'$.
    By~\Cref{lem:key-step}, there is an inclusion representation $\cgS' = \inrep{S_x' : x \in P}$ of $P$ with $|S_x'| \leq |S_x|$ for every $x \in P$ and
        \[\left|\bigcup\cgS'\right| \leq \iir(Q') + \varepsilon + |S_y| < (\iir(Q') + \varepsilon) + |D_P[y]|.\]
    Note that by~\Cref{thm:iir-bound}, $\iir(Q') \leq |Q'|$, which when $\varepsilon = 0$ yields $|\bigcup \cgS'| < |P|$.
    In particular, $\cgS'$ is a strict reduction of $\cgS$: a contradiction.
    On the other hand, when $\varepsilon = 1$, by~\Cref{cor:irr-unique-minimal-element}, $\iir(Q') \leq |Q'|-1$, which gives the same contradiction.
\end{proof}

\begin{corollary}\label{cor:miir-canonical}
  For every poset $P$, $\iir(P)=|P|$ if and only if the canonical inclusion representation of $P$ is irreducible.
\end{corollary}
\begin{proof}
    The canonical inclusion representation $\cgC$ of a poset $P$ always satisfies $|\bigcup \cgC| = |P|$, hence, if $\cgC$ is irreducible, then $\iir(P) \geq |P|$, and so, $\iir(P) = |P|$ by~\Cref{thm:iir-bound}.
    On the other hand, if for a poset $P$, we have $\iir(P)=|P|$, then let us fix an inclusion representation $\cgS$ of $P$ with $|\bigcup \cgS| = |P|$.
    Then, by~\Cref{cor:irr-IR-downsets}, for every $x \in P$, we have $|S_x| = |D_P[x]|$.
    In particular, $\cgS$ is a reduction of $\cgC$ and $\cgC$ is a reduction of $\cgS$, thus, they are equivalent.
\end{proof}
Later (\Cref{cor:miir-canonical-isomorphism}), we prove that for a poset $P$ with $\iir(P) = |P|$, the canonical inclusion representation of $P$ is the only irreducible inclusion representation of $P$ with the ground set of size $|P|$ up to the isomorphism.

\medskip
\section{Combining and splitting inclusion representations} \label{sec:combining-splitting}
\medskip

Let $t$ be an integer with $t\ge2$, and let $(Q_1,\dots,Q_t)$ be a sequence 
of posets with disjoint ground sets.  Then  are (at least) two
natural ways to combine these posets into a larger poset.
In both cases, the ground set of the new poset is the union of the 
ground sets of the posets in the sequence.

First, the \defin{disjoint sum} of $(Q_1,\dots,Q_t)$, denoted \defin{$P=Q_1+\dots+Q_t$}, 
is the poset such that if $x,y\in P$, then we have $x\le y$ in $P$ if and 
only if there is some $i\in[t]$ such that $x,y\in Q_i$ and $x\le y$ in~$Q_i$. 
We say that a poset is a \defin{component}\footnote{In the literature, researchers
typically classify a poset as \defin{connected} or \defin{disconnected}.
A disconnected poset is the disjoint sum of its components.  We elect to use the more
compact terminology of simply refering to connected posets as components.}
when there does not exist non-empty
subposets $Q_1$ and $Q_2$ such that $P=Q_1+Q_2$.
It follows that when $P$ is not a component, there is a uniquely determined
integer $t$ with $t\ge2$ for which $P=Q_1+\dots+Q_t$, and for each $i\in[t]$, $Q_i$ 
is a component.  In this case, we refer to the expression $Q_1+\dots+Q_t$ as the 
\defin{component decomposition} of $P$.
Also, the posets in the sequence $(Q_1,\dots,Q_t)$ are called \defin{components of $P$}.


The \defin{vertical sum}\footnote{Vertical sums have appeared in the literature, 
and have also been called \defin{linear sums} and \defin{joins}. 
Both disjoint sum and vertical sums are special cases of a \defin{lexicographic 
sum}, but there does not appear to be any application of the more general 
definition in this setting.} of $(Q_1,\dots,Q_t)$,
denoted \defin{$P=Q_1<\dots<Q_t$} is the poset such that we have $x\le y\in P$ if and 
only if one of the following two conditions holds: 
(1)~there is some $i\in[t]$ with $x,y\in Q_i$ and $x\le y$ in $Q_i$; 
(2)~there are integers $i,j\in[t]$ with $i<j$ such that $x\in Q_i$ and $y\in Q_j$.  
We say that a poset $P$ is a \defin{vertical prime} when there does not 
exist posets $Q_1$ and $Q_2$ such that $P=Q_1<Q_2$.  

We say that a poset $P$ is 
a \defin{block} when it is either a chain or a vertical prime.  
It follows that when $P$ is not a block, there is a least integer $t$ with $t\ge2$ for 
which $P=Q_1<\dots<Q_t$, and for each $i\in[t]$, $Q_i$ is a block.   
In this case, we refer to the expression $Q_1<\dots<Q_t$ as the 
\defin{block decomposition} of $P$.  Also, the posets in the sequence 
$(Q_1,\dots,Q_t)$ are called \defin{blocks of $P$}.
Returning to Figure~\ref{fig:cw-fig1}, the poset $P$ has three blocks, while the
poset $Q$ has two blocks.

We state the following elementary result for emphasis.

\begin{proposition}\label{prop:poly-block-decomposition}
  There exists a polynomial time algorithm that for a poset $P$ 
  returns the block decomposition of $P$.
\end{proposition}


Next, we discuss how inclusion representations of components and blocks 
of a poset $P$ relate to inclusion representations of $P$.

\subsection{Inclusion representations of components}
Let $P$ be a poset that is not a component and let $P=Q_1+\dots+Q_t$ be 
the component decomposition of $P$.
Given an inclusion representation of $P$, we construct an inclusion 
representation of each component of $P$.
Assume that $\cgS = \inrep{S_x:x\in P}$ is an inclusion representation of $P$.
Note that all sets in $\cgS$ are non-empty.  
For each $i\in[t]$ and $x\in Q_i$, let $S_{i,x} = S_x$.
It follows that $\cgS_i = \inrep{S_{i,x}:x \in Q_i}$ is an inclusion representation 
of $Q_i$ for every $i \in [t]$.  We will refer the inclusion representations 
$\cgS_1,\dots,\cgS_t$ as the \defin{components} of $\cgS$.

Next, we show how to construct an inclusion representation of $P$ given 
inclusion representations of $Q_i$ for each $i \in [t]$.
For each $i\in[t]$, let $\cgS_i = \inrep{S_{i,x}:x \in Q_i}$ be an inclusion 
representation of $Q_i$.  Without loss of generality, we can assume that for all $i,j \in [t]$,
$\bigcup\cgS_i$ and $\bigcup\cgS_j$ are disjoint whenever $i \neq j$.  The most naive thing to do is just to assign 
$S_{i,x}$ to $x \in P$ where $x \in Q_i$.  However, such an assignment results
in an inclusion representation of $P$ if and only if all the sets are non-empty.
In the general case, we need something slightly more sophisticated.

Let $m$ denote the number of components of $P$ for which $\emptyset$ is one of the
sets in $\cgS_i$.  Up to a simple relabeling, we may assume that 
these components are $Q_1,\dots,Q_m$.  We have already noted that for each $i\in[m]$,
this implies that $Q_i$ has a unique minimal element. 
Choose an $m$-element set $\{a_1,\dots,a_m\}$ which is disjoint from
$\bigcup\cgS_j$ for all $j\in[t]$. 
Now, for every $x \in P$ with $i\in[t]$, we set $S_x = S_{i,x}$ when 
$i > m$ and $S_x = S_{i,x}\cup\{a_i\}$ when $i\leq m$.
Observe that $\cgS = \inrep{S_x:x \in P}$ is an inclusion representation of $P$.
In the remainder of the paper, we will write \defin{$\cgS=\cgS_1+\dots+\cgS_t$}, when
$\cgS$ has been constructed in this manner.

We conclude the discussion on disjoint sums with the following proposition that 
now follows immediately.

\begin{proposition}\label{lem:components}
  Let $P$ be a poset, which is not a component, and let 
  $P = Q_1+\dots+Q_t$ .  Let $m$ be the number of $i \in [t]$ such that 
  $Q_i$ has a unique minimal element.  When $m>0$, we assume that these components
  are $Q_1,\dots,Q_m$.  Then the following statements hold:
  \begin{enumerate}
    \item $\ch(P)=h_0$ if $\ch(Q_i) < h_0$ for all $i \in [m]$. 
      where $h_0 = \max\{\ch(Q_i): i \in [t]\}$; otherwise, $\ch(P)=h_0+1$.
    \item $\dim_2(P)\le \dim_2(Q_1)+\dots+\dim_2(Q_t)+m$.
    \item $\cw(P)\le \cw(Q_1)+\dots+\cw(Q_t)+m$.
    \item $\iir(P)\leq \iir(Q_1)+\dots+\iir(Q_t)+m$.\label{lem:components:msiir}
  \end{enumerate}
\end{proposition}

\subsection{Inclusion representations of blocks} \label{sec:IR-blocks}
Let $P$ be a poset that is not a block, and let $P=Q_1<\dots<Q_t$ be the 
block decomposition of $P$.
First, let $\cgS=\inrep{S_x:x\in P}$ be an inclusion representation of $P$. We use
$\cgS$ to construct an inclusion representation of each block of $P$.
For every $i \in [t]$, let $W_i = \emptyset$ when $i = 1$; otherwise,
set $W_i = \bigcup\{S_y: y\in Q_{i-1}\}$.

Now let $x\in P$, and let $i$ be the unique integer in $[t]$ such that $x\in Q_i$. 
Define $S_{i,x} = S_x - W_i$.  Observe that for each $i\in[t]$, $\cgS_i = 
\inrep{S_{i,x}:x\in Q_i}$ is an inclusion representation of $Q_i$.

Next, we show how to construct an inclusion representation of $P$ 
given inclusion representations of $Q_i$ for each $i \in [t]$. For each $i\in[t]$,
let $\cgS_i = \inrep{S_{i,x}:x\in Q_i}$ be an inclusion representation of $Q_i$.
Without loss of generality, we can assume that $\bigcup\cgS_i$ and $\bigcup\cgS_j$ are 
disjoint whenever $i$ and $j$ are distinct integers in $[t]$.
We define an inclusion representation $\cgS=\inrep{S_x:x\in P}$ of $P$ as follows.
Let $x\in P$, and $i$ be the unique integer in $[t]$ such that $i\in Q_i$.  Then 
set:
\[
  S_x=S_{i,x}\cup\bigcup\{\cgS_j:j \in [i-1]\}. 
\]
It is easy to see that $\cgS = \inrep{S_x: x\in P}$ is an inclusion 
representation of $P$.  Furthermore, $\cgS$ is irreducible if and only 
if $\cgS_i$ is irreducible for every $i\in [t]$.

We again conclude the discussion of vertical sums with a nearly self-evident proposition.

\begin{proposition}\label{lem:vertical-sum}
  Let $P$ be a poset which is not a block, and let $P=Q_1<\dots<Q_t$ be
  the block decomposition of $P$.
  Then,
  \begin{enumerate}
    \item $\ch(P)=\dim_2(P-Q_t)+\ch(Q_t)$,
    \item $\dim_2(P)=\dim_2(Q_1)+\dots+\dim_2(Q_t)$,\label{lem:vertical-sum:dim2}
    \item $\cw(P)=\dim_2(P-Q_t)+\cw(Q_t)$,\label{lem:vertical-sum:cw}
    \item $\iir(P)=\iir(Q_1)+\dots+\iir(Q_t)$. \label{lem:vertical-sum:iir}
  \end{enumerate}
\end{proposition}

For posets expressible as certain vertical sums, we can easily get better bounds than in~\Cref{thm:iir-bound} using~\Cref{lem:vertical-sum}.\ref{lem:vertical-sum:iir}.

\begin{corollary}
    If $P$ is a poset with the block decomposition $Q_1<\dots<Q_t$ and $m$ is the number of blocks of $P$ that are chains, then $\iir(P) \leq |P|-m$.
\end{corollary}

\medskip
\section{Characterization problems: \texorpdfstring{$\iir$}{iir}}\label{sec:char-thms-iir}
\medskip

Let $P$ be a poset and consider the following properties, each of which $P$ may or
may not satisfy:
\begin{itemize}
    \item \defin{No Block is a Chain Property:} If $x$ is an element of $P$, then there is an element $y \in P$ such that $x \parallel y$ in $P$.
    \item \defin{Two Down Property:} If $y$ is an element of $P$, and $y$ covers at least two distinct elements of $P$, then there is an element $z$ of $P$ such that $D_P(y)\subseteq D_P(z)$ and $y\parallel z$ in $P$.
    \item \defin{Parallel Pair Property:} If $x$ and $y$ are incomparable elements of $P$, then at least one of the following two statements holds: 
    (1)~there is an element $y'\in P$ with $D_P(x)\subseteq D_P(y')$, $y\le y'$ in $P$, and $x\parallel y'$ in $P$; 
    (2)~there is an element $x'$ of $P$ such that $D_P(y)\subseteq D_P(x')$, $x\le_P x'$ in $P$, and $y\parallel_P x'$ in $P$. 
\end{itemize}
Note that each of the properties for a given poset can be verified in polynomial time by simple brute-force algorithms.
The goal of this section is to prove the following result that yields~\Cref{thm:poly} for $p = \iir$.

\begin{theorem}\label{thm:char:iir}
    For a poset $P$, we have $\iir(P) = |P|$ if and only if $P$ satisfies the No Block is a Chain Property, the Two Down Property, and the Parallel Pair Property.
\end{theorem}

First, we illustrate the properties with some examples.
We again refer to~\Cref{fig:cw-fig1}.
Note that the poset $P$ violates
all three of these properties.  The element associated with the set
$\{1,2,3\}$ is comparable with all other elements of $P$;  the elements
associated with the sets $\{1,2,3\}$ and $\{1,2,3,4,5,6\}$ violate the
Two Down Property; and the elements associated with $\{1,2,3,4,6\}$ and
$\{1,2,3,5,6\}$ violate the Parallel Pair Property.  On the other hand, the
poset $Q$ does indeed satisfy all three.  Again, we note that $\iir(Q)=|Q|=8$.

\begin{lemma}\label{lem:MIIR-3props-nec}
  For a poset $P$, if $\iir(P)=|P|$, then $P$ satisfies the
  No Block is a Chain Property, the Two Down Property, and the
  Parallel Pair Property.
\end{lemma}

\begin{proof}
  We have three statements to prove, and each will be handled using an
  argument by contradiction.
  Let $P$ be a poset such that $\iir(P)=|P|$.

  Suppose first that $P$ does not satisfy the No Block is a Chain Property.
  Let $x$ be an element of $P$ for which there is no element $y$ of $P$ with
  $x\parallel y$ in $P$.  Let $P=Q_1<\dots<Q_t$ be the block decomposition of
  $P$, and let $i$ be the integer in $[t]$ such that $x\in Q_i$.  Then
  $Q_i$ is a chain. For a chain, $\iir(Q_i) = |Q_i| - 1 < |Q_i|$, and
  $\iir(P)=\sum_{i\in[t]} \iir(Q_i)$, it follows that $\iir(P)<|P|$. 
  The contradiction proves that $P$ must satisfy the No Block is a Chain Property.

  Now suppose that $P$ does not satisfy the Two Down Property.
  There is an element $y$ of $P$ such that $y$ covers at least 
  two elements of $P$ but there is no element $z$ of $P$ with $D_P(y)\subseteq 
  D_P(z)$ and $y\parallel z$ in $P$.  We form an inclusion representation 
  $\cgS=\inrep{S_x:x\in P}$ of $P$ as follows.  First, set $S_y=D_P(y)$.  
  Then for each $x\in P$ with $x\neq y$, set $S_x=D_P[x]$. 
  Since $y$ violates the Two Down Property, $\cgS$ is an inclusion representation of $P$.
  Moreover, $\cgS$ is a strict reduction of the canonical inclusion representation of $P$. 
  Therefore, \cref{cor:miir-canonical} implies that $\iir(P)<|P|$.
  This is a contradiction, hence, $P$ satisfies the Two Down Property.

  Finally, suppose that $P$ does not satisfy the Parallel Pair Property.
  Then there is an incomparable pair $x,y$ of 
  elements of $P$ for which neither of the two statements of the Parallel 
  Pair Property holds.  We form an
  inclusion representation $\cgT=\{T_u: u \in P\}$ of $P$ using the 
  following rules.  If $u\in P$ and $u\not\ge y$ in $P$, set $T_u=D_P[u]$; 
  if $u\ge y$ in $P$, set $T_u=\{x\}\cup (D_P[u]-\{y\})$. Now $\cgT$ is an inclusion representation of $P$, and note that $\bigcup \cgT$ does not contain $y$, therefore, it is a strict reduction of $\cgC$.  Again, \cref{cor:miir-canonical} implies $\iir(P)<|P|$.
  The contradiction proves that $P$ must satisfy the Parallel Pair Property,
  and with this observation, the proof of the lemma is complete.
\end{proof}

Now we turn our attention to showing that the three properties are sufficient.  
The argument requires a preliminary lemma.

\vbox{
\begin{lemma}\label{lem:3-props-downset}
  Let $P$ be a poset that satisfies the No Block is a Chain Property and the Parallel Pair Property.  Let 
  $\cgS=\inrep{S_x:x\in P}$ be an inclusion representation of 
  $P$. If $Q$ is a down set in $P$, $\cgT=\inrep{S_x:x\in Q}$, and $|S_x|=|D_P[x]|$ for every $x\in Q$,
  then $\cgT$ is isomorphic to the canonical inclusion representation
  of $Q$. 
\end{lemma}
}
\begin{proof}
  We argue by contradiction.  A counterexample is a triple $(P,\cgS,Q)$ for which
  the hypothesis is satisfied, but $\cgT$ is not isomorphic to the canonical
  inclusion representation of $Q$.  Of all counterexamples,
  we choose one for which $|Q|$ is minimum.

  We note that since $P$ satisfies the No Block is a Chain
  Property, it has at least two minimal elements.  Therefore, in any
  inclusion representation of $P$, all sets are non-empty.
  In particular, the statement of the lemma holds when $|Q| = 1$, and so, from now on, we assume that in the chosen counterexample $|Q| \geq 2$.

  Let $y$ be a maximal element of $Q$ and let $Q' = Q - \{y\}$.
  We have $|Q'| < |Q|$, hence, the assertion of the lemma holds for $(P,\cgS,Q')$, and so, $\cgT' = \inrep{S_x : x \in R}$ is isomorphic to the canonical inclusion representation of $Q'$.
  By relabeling if necessary, we can assume that the elements in $\bigcup \cgS$ are labeled so that $S_x = D_P[x]$ for every $x \in Q'$.

  By assumption, $|S_y| = |D_P[y]|$, and so, $|S_y| > |D_P(y)|$.
  Therefore, there is a unique element $x \in S_y - D_P(y)$.
  If $x \notin Q - \{y\}$, then we relabel $x$ to be $y$ in $\cgS$ and we obtain that $\cgT$ is isomorphic to the canonical inclusion representation of $Q$.
  However, we assumed that this is false, and also, $x \notin D_P(y)$, hence, $x \in Q - D_P[y]$.
  Since $y$ is a maximal element in $Q$, we have $y \not< x$ in $P$.
  Thus, $x \parallel y$ in $P$.
  By the Parallel Pair Property, one of the following two statements holds: 
   (1)~there is an element $y'\in P$ with $D_P(x)\subseteq D_P(y')$, $y\le y'$ in $P$, and $x\parallel y'$ in $P$; 
   (2)~there is an element $x'$ of $P$ such that $D_P(y)\subseteq D_P(x')$, $x\le_P x'$ in $P$, and $y\parallel_P x'$ in $P$. 
  When (1) holds, we have $x \in S_y \subset S_{y'}$, and so, $S_x = D_P[x] \subset S_{y'}$, which is a contradiction with $x \parallel y'$ in $P$.
  When (2) holds, we have $x \in D_P[x] = S_x \subset S_{x'}$, and so, $S_y = D_P(y) \cup \{x\} \subset S_{x'}$, which is a contradiction with $y \parallel x'$ in $P$.
    These contradictions complete the proof of the lemma.
\end{proof}

\begin{proof}[Proof of~\Cref{thm:char:iir}]
  We have already shown (\Cref{lem:MIIR-3props-nec}) that the three properties 
  are necessary for $\iir(P)=|P|$ to hold.
  It remains only to show that they
  are sufficient.  We argue by contradiction.  Let $P$ be a poset that
  satisfies all three properties but $\iir(P)<|P|$.
  Let $\cgC=\inrep{D_P[x]:x\in P}$ be the canonical inclusion representation
  of $P$.
  \Cref{cor:miir-canonical} ensures the existence of $\cgS=\inrep{S_x:x\in P}$ an inclusion representation
  of $P$ that is a strict reduction of $\cgC$.  
  By definition of reduction, we have $|\bigcup\cgS|\le|P|$ and $|S_x|\le|D_P[x]|$ for every $x\in P$.
  If $|S_x| = |D_P[x]|$ for every $x \in P$, then by~\Cref{lem:3-props-downset} applied to $Q = P$, $\cgS$ is isomorphic to $\cgC$, which contradicts with the reduction being strict.
  Therefore, there is $v \in P$ such that $|S_v| < |D_P[v]|$.
  Let $V$ be the (non-empty) set consisting of all such $v \in P$ and let $y \in P$ be a minimal element of $V$.

  Since $P$ satisfies the No Block is a Chain Property, it has at least two
  minimal elements.  Accordingly, all sets in $\cgS$ are non-empty.
  Therefore $y$ is not a minimal element of $P$.  However, since $y$ is
  minimal element of $V$, we know that $|S_x|=|D_P[x]|$ for every $x\in D_P(y)$.
  By~\Cref{lem:3-props-downset} applied to $Q=D_P(y)$, it
  follows that the restriction of $\cgS$ to $D_P(y)$ is isomorphic to the canonical inclusion representation of $D_P(y)$.  
  By renaming elements in the representation, without loss of generality, we can assume that $S_x = D_P[x]$ for every $x \in D_P(y)$.
  Since $S_x\subseteq S_y$ for every $x\in D_P(y)$, it follows that $D_P(y)\subseteq S_y$.
  In particular, $|D_P(y)|\le |S_y|<|D_P[y]| = |D_P(y)| + 1$, hence, $S_y=D_P(y)$.

  If $y$ covers exactly one element in $P$, say $u$,
  then $S_u=D_P[u] = D_P(y)= S_y$, which is impossible.  
  We conclude that $y$ covers at least two elements in $P$.
  Since $P$ satisfies the Two Down Property, there
  is an element $z\in P$ such that $y\parallel z$ in $P$, and $D_P(y)\subseteq
  D_P(z)$.  
  In particular, we have $D_P(y) \subset S_z$.
  However, this forces $S_y\subseteq S_z$, which contradicts the fact that $y$ and $z$ are incomparable.
  With this observation, the proof of the theorem is complete.
\end{proof}

As promised, we strengthen~\Cref{cor:miir-canonical} by showing that for a poset $P$ with $\iir(P) = |P|$, the canonical inclusion representation of $P$ is the only irreducible inclusion representation of $P$ with the ground set of size $|P|$ up to the isomorphism.

\begin{corollary}\label{cor:miir-canonical-isomorphism}
  Let $P$ be a poset with $\iir(P) = |P|$ and let $\cgS=\inrep{S_x: x\in P}$ be an irreducible inclusion representation
  of $P$ with $|\bigcup\cgS|=|P|$.
  Then, $\cgS$ is isomorphic to
  the canonical inclusion representation of $P$.
\end{corollary}

\begin{proof}
    By~\Cref{cor:irr-IR-downsets}, $|S_x| = |D_P[x]|$ for every $x \in P$.
    By~\Cref{thm:char:iir}, $P$ satisfies the No Block is a Chain Property and the Parallel Pair Property.
    Applying~\Cref{lem:3-props-downset} with $Q = P$, we obtain that $\cgS$ is isomorphic to the canonical inclusion representation of $P$.
\end{proof}

\subsection{More structural insight into posets with \texorpdfstring{$\iir(P) = |P|$}{iir(P) = |P|}} \label{sec:structural}
We close this section with two technical results that will prove useful in the other two characterization problems.
Let \defin{$\MIIR$} (\defin{maximum irreducible inclusion representation} posets) be the class of all posets $P$ with $\iir(P) = |P|$.
Next, let \defin{$\NMIIR$} (\defin{nearly maximum irreducible inclusion representation} posets) be the class of all posets $Q$ such that 
\begin{enumerate}[label=(\roman*)]
    \item $\iir(Q)=|Q|$ or
    \item $\iir(Q)=|Q|-1$ and if $Q=R_1<\dots<R_s$ is the block decomposition of $Q$, then $R_1$ is a chain. \label{item:NMIIR:ii}
\end{enumerate}
Recall that for every chain $C$, we have $\iir(C) = |C|-1$, hence,~\ref{item:NMIIR:ii} implies that for every $i  \in \{2,\ldots,s\}$, we have $\iir(R_i)=|R_i|$.  
Observe that all chains belong to $\NMIIR$. 
Also, observe that posets in $\NMIIR$ with at least two minimal elements are in $\MIIR$

\begin{lemma}\label{pro:miir-upset}
  Let $P$ be a poset and let $Q$ is a non-empty up set of $P$.
  If $\iir(P)=|P|$, then $Q$ is in $\NMIIR$.
\end{lemma}

\begin{proof}
  We argue by contradiction.  A counterexample is
  a pair $(P,Q)$ where $P$ is a poset with $\iir(P)=|P|$, 
  $Q$ is a non-empty up set of $P$, and $Q$ does not
  belong to $\NMIIR$.  We choose a counterexample $(P,Q)$ such
  that $|Q|$ is minimum.
  Since $\iir(P)=|P|$, $\cgC$ is irreducible by~\Cref{cor:miir-canonical}.  

  Suppose first that $Q$ has at least two minimal elements.  
  Since $Q$ does not belong to $\NMIIR$, we know $\iir(Q)<|Q|$.  It
  follows that there is an inclusion representation $\cgT=\inrep{T_v:v\in Q}$ 
  of $Q$, which is a strict reduction of the canonical inclusion 
  representation for $Q$.  Note that all sets in $\cgT$ are non-empty.
  Also, without loss of generality, assume that the $\bigcup \cgT$ is disjoint from $P$.
  We form an inclusion representation $\cgS=\inrep{S_x:x\in P}$ of $P$
  using the following rules. 
  If $x\in P-Q$, then $S_x=D_P[x]$;
  and if $v\in Q$, then $S_v=T_v\cup [D_P(v)\cap (P-Q)]$.
  Since all the sets in $\cgT$ are non-empty, $\cgS$ is indeed an inclusion representation of $P$.
  Moreover, since $\cgT$ is a strict reduction of the canonical inclusion representation of $Q$, it follows that $\cgS$ is a strict reduction of the canonical inclusion representation of $P$, which is a contradiction, implying that $Q$ has a unique minimal element.  

  Let $Q=R_1<\dots<R_s$ be the block decomposition of $Q$.  
  In particular, $R_1$ is a chain.  
  If $s=1$, then $Q$ is a chain, and so, $Q$ is in $\NMIIR$.  
  It follows that $s\ge2$.   Furthermore, $Q-R_1$ is an up set of $P$.  By minimality of $(P,Q)$ the lemma holds for the pair 
  $(P,Q-R_1)$, therefore $Q-R_1$ is in $\NMIIR$. 
  Note that $Q-R_1$ has at least two minimal elements, hence, $Q - R_1 \in \MIIR$ and thus $Q \in \NMIIR$.  The contradiction completes the 
  proof.
\end{proof}

\begin{lemma}\label{pro:miir-components}
  Let $P$ be a poset that is not a component.
  If $\iir(P)=|P|$, then at most one component of $P$ is 
  non-trivial and all components of $P$ are in $\NMIIR$.
\end{lemma}

\begin{proof}
  Each component of $P$ is an up set of $P$, and therefore belongs
  to $\NMIIR$. 
  Now, suppose that there are two distinct non-trivial components $Q$ and $Q'$ of $P$.
  Let $y$ be a maximal element of $Q$ and let $y'$ be a maximal
  element of $Q'$. 
  The pair $(y,y')$ shows that $P$ does not satisfy the Parallel Pair Property.  
  This contradicts~\Cref{thm:char:iir} and completes the
  proof.
\end{proof}


\medskip
\section{Characterization problems: \texorpdfstring{$\dim_2$}{dim2} and \texorpdfstring{$\cw$}{cw}}
\label{sec:dimcw}
\medskip

In this section, we complete the proof of~\Cref{thm:poly}.
More precisely, for each parameter $p \in \{\dim_2,\cw\}$ we give a polynomial time algorithm that for every poset $P$ decides if $p(P) = |P|$.
Let us first state the criteria that we later prove.
To this end, we make the following definitions:
\begin{itemize}[label = $\circ$]
  \item let $Z$ be a four elements poset with three components (unique up to isomorphism);
  \item let $\bbA$ consists of all non-trivial antichains;
  \item let $\bbA_{2,3,4}$ consists of antichains of sizes $2$, $3$ and $4$;
  \item let $\bbB$ consists of all posets of the form $C+T$, where $C$ is a non-trivial chain and $T$ is a trivial poset;
  \item let $\MTD$ be the family of all posets $P$ such that if $P=Q_1<\dots<Q_t$ is the block decomposition of $P$, then $Q_i$ is a poset in $\bbA_{2,3,4}\cup\bbB\cup\{Z\}$, for every $i\in[t]$;
  \item let $\MCW$ be the family of all posets $P$ such that if $P=Q_1<\dots<Q_t$ is the block decomposition of $P$, then $Q_t$ is in $\bbA\cup\bbB\cup\{Z\}$, and, if $t>1$, then $P-Q_t$ is in $\MTD$.
\end{itemize}

\begin{theorem}\label{thm:char:dim2}
    For a poset $P$, we have $\dim_2(P) = |P|$ if and only if $P \in \MTD$.
\end{theorem} 

\begin{theorem}\label{thm:char:cw}
    For a poset $P$, we have $\cw(P) = |P|$ if and only if $P \in \MCW$.
\end{theorem}

Recall that finding the block decomposition of a poset can be done in polynomial time.
In particular, testing for being in~$\MTD$ or $\MCW$ can done in polynomial time, hence Theorems~\ref{thm:char:dim2} and~\ref{thm:char:cw} imply~\Cref{thm:poly} for $p \in \{\dim_2,\cw\}$.

Let us start an immediate consequence of~\Cref{lem:vertical-sum} that allows us to restrict only to posets that are blocks.

\begin{proposition}
  Let $P$ be a poset which is not a block, and let $P=Q_1<\dots<Q_s$ be 
  the block decomposition of $P$.  Then the following statements hold:
  \begin{enumerate}
    \item $\dim_2(P)=|P|$ if and only if $\dim_2(Q_i)=|Q_i|$ for
      all $i\in[s]$;
    \item $\cw(P)=|P|$ if and only if $\dim_2(Q_i)=|Q_i|$ for
      all $i\in[s-1]$, and $\cw(Q_s)=|Q_s|$.
  \end{enumerate}
\end{proposition}

Now, the necessity of the conditions in Theorems~\ref{thm:char:dim2} and~\ref{thm:char:cw} follows from the following elementary result that can be easily verified.

\begin{proposition}\label{pro:2d-cw-basic}
  The following two statements hold.
  \begin{enumerate}
    \item If $P$ is a poset in $\bbA_{2,3,4}\cup\bbB\cup\{Z\}$,
      then $P$ is a block and $\dim_2(P)=|P|$.
    \item If $P$ is a poset in $\bbA\cup\bbB\cup\{Z\}$,
      then $P$ is a block and $\cw(P)=|P|$.
  \end{enumerate} 
\end{proposition} 

In the final proof we will need the following technical detail on cube height.

\begin{lemma}\label{pro:cw-ch}
Let $P$ be a poset.
  If $\cw(P)=|P|$, then $\ch(P)=
  \max\{|D_P[x]|:x\in P\}$.
\end{lemma}

\begin{proof}
  Arguing by contradiction, we assume that $\cw(P)=|P|$ and $\ch(P)<
  \max\{|D_P[x]|:x\in P\}$.  Let $y$ be an element of $P$ with 
  $|D_P[y]|=\max\{|D_P[x]|:x\in P\}$.  Then let $\cgS=
  \inrep{S_x:x\in P}$ be an irreducible inclusion representation of $P$ with
  $|\bigcup\cgS|= \cw(P) = |P|$ and $|S_x|\le\ch(P)$ for every $x\in P$. 
  In particular, $|S_y|<|D_P[y]|$.  
  \Cref{cor:irr-IR-downsets} now forces
  $|\bigcup\cgS|<|P|$.  The contradiction completes the proof.
\end{proof}

We note that~\Cref{pro:cw-ch} may not hold for
a poset $Q$ for which $\cw(Q)<\iir(Q)=|Q|$, as evidenced by
the poset $Q$ shown on the right side of
Figure~\ref{fig:cw-fig1}.  We note that
$\ch(Q)=5$, $\max\{|D_Q[x]|:x\in Q\}=7$, and $\iir(Q)=|Q|=8$.

The following theorem completes the proofs of Theorems~\ref{thm:char:dim2} and~\ref{thm:char:cw}.

\begin{theorem}\label{thm:dim2-cw-block}
  Let $P$ be a poset that is a block.  Then
  the following statements hold.
  \begin{enumerate}[label=\rm{(S\arabic*)}, ref={Statement (S\arabic*)}]
    \item If $\dim_2(P)=|P|$, then $P\in\bbA_{2,3,4}\cup\bbB\cup\{Z\}$. \label{item:dim2}
    \item If $\cw(P)=|P|$, then $P\in\bbA\cup\bbB\cup\{Z\}$. \label{item:cw}
  \end{enumerate} 
\end{theorem} 

\begin{proof}
  We argue by contradiction.  Let $P$ be a poset that is a block for
  which (at least) one of the two statements of the theorem fails.

  We note that if $\dim_2(P)=|P|$, then $\cw(P)=|P|$.  Also, if
  $\cw(P)=|P|$, then $\iir(P)=|P|$. 
  Therefore by~\Cref{thm:char:iir}, $P$ satisfies the No Block is a Chain Property, the Two Down Property, and the Parallel Pair Property.
  Also, using~\Cref{pro:cw-ch},
  we know that $\ch(P)=\max\{|D_P[x]|:x\in P\}$. 

  \begin{claim*}
    $P$ is neither an antichain nor a chain.
  \end{claim*}

  \begin{proofclaim}
    An antichain can not be a counterexample to~\ref{item:cw}.
    Moreover, recall that if $A$ is an $n$-element antichain then $\dim_2(A) < |A|$ unless $|A| \in \{2,3,4\}$, thus, an antichain is also not a counterexample to~\ref{item:dim2}.
    For every chain $C$, we have $\iir(C) = |C|-1$, hence, a chain is not a counterexample to any of the statements.
  \end{proofclaim}

  Since $P$ is a block, which is not a chain, it has at least two maximal elements and at least two
  minimal elements.  However, we do not know whether $P$ is a component or not.

  \begin{claim*}\label{lem:P-connected}
    $P$ is a component.
  \end{claim*}

  \begin{proofclaim}
    We argue by contradiction, assuming that $P$ is not a component.
    Let $P=Q_1+\dots+Q_t$ be the component decomposition of $P$ (note that $t \geq 2$).
    Since $\iir(P)=|P|$ and $P$ is not an antichain, using \cref{pro:miir-components}, we may assume
    that $Q_1,\dots,Q_{t-1}$ are trivial, and $Q_t$ is a non-trivial
    poset in $\NMIIR$. For each $i\in[t-1]$, we let $u_i$ be the
    singleton element in~$Q_i$.

    \begin{subclaim*}\label{clm:Qt-chain}
      $Q_t$ is a chain.
    \end{subclaim*}

    \begin{proofsubclaim}
      Let $Q_t = R_1 < \dots < R_s$ be the block decomposition of $Q_t$.
      If $s = 1$, then the statement follows, hence, assume to the contrary that $s \geq 2$.
      Recall that since $Q_t \in \NMIIR$, $\iir(R_i) = |R_i|$ for every $i \in \{2,\dots,s\}$.
      In particular, $R_s$ is not a chain.
      Let $y$ and $z$ be distinct
      maximal elements of $Q_t$ and remark that $\ch(P) \geq 2$ since $Q_t$ is non-trivial.  We form an inclusion representation $\cgS=\inrep{S_x:x\in P}$ of $P$ using the
      following rules. Set $S_{u_1}= \{y,z\}$; $S_{u_j}=\{u_j\}$ if 
      $j\in \{2,\dots, t-1\}$; and $S_x=D_P[x]$ if $x\in Q_t$ .
      Note that $\cgS$ is indeed an inclusion representation of $P$.
      Moreover, since $\ch(P) \geq 2$ and $\ch(P) = \max\{|D_P[x]| : x \in P\}$ (\Cref{pro:cw-ch}), we have $|S_x| \leq \ch(P)$ for every $x \in P$.
      Finally, since $u_1 \notin \bigcup \cgS$, we obtain $|\bigcup \cgS| < |P|$, yielding $\cw(P) < |P|$, which shows that $P$ is not a counterexample assuming that $Q_t$ is not a chain.
    \end{proofsubclaim}

    We may assume that there are no integers in $\{u_1,\dots,u_{t-1}\}$.
    Let $n=|Q_t|$.  Then $n\ge2$, and we may assume that the elements of 
    $Q_t$ are labeled with the integers in $[n]$ with $i<j$ in $Q_t$ 
    if and only if $i<j$ as integers.  
    Again note that $\ch(P) \geq 2$ in this case.

    If $t=2$, then $P\in\bbB$.
    If $t = 3$ and $n = 2$, then $P = Z$.
    In both cases above, $P$ is not a counterexample to \ref{item:dim2} or \ref{item:cw}.
    We split the remaining cases into two: (Case 1) $t \geq 4$ and (Case 2) $t = 3$ and $n \geq 3$.

    In each of these two cases, we will reach a 
    contradiction by constructing an inclusion representation $\cgS = \inrep{S_x : x \in P}$ such that $|S_x| \leq \ch(P)$ for every $x \in P$ and $|\bigcup \cgS| < |P|$.
    Similarly as in the Subclaim above, this will yield $\cw(P) < |P|$ showing that $P$ is not a counterexample to the statement.
    of $P$.

    In Case 1, we use the following rules:
    \begin{itemize}[label = $\circ$]
      \item set $S_{u_1}=\{u_2,u_3\}$;
      \item set $S_{u_2}=\{u_2,n\}$; 
      \item set $S_{u_3}=\{u_3,n\}$;
      \item set $S_x=D_P[x]$ for each $x\in P-\{u_1,u_2,u_3\}$.
    \end{itemize} 
    
    In Case 2, we use the following rules:
    \begin{itemize}[label = $\circ$]
      \item set $S_{u_1}=\{u_2,n-1\}$;
      \item set $S_{u_2}=\{u_2,n\}$; 
      \item set $S_x=D_P[x]$ for each $x\in P-\{u_1,u_2\}$.
    \end{itemize} 
    It is easy to verify that $\cgS = \inrep{S_x : x \in P}$ satisfies the required conditions in both cases.
    With these observations, the proof of the claim
    is complete.
  \end{proofclaim}

  Now we have shown that $P$ is both a block and a component.
  Let $M$ denote the set of maximal elements of $P$. 
  Since $P$ is not a chain, $|M|\ge2$. 
  The next claim is an immediate consequence of the fact that
  $P$ satisfies the Parallel Pair Property.

  \begin{claim*}
    If $y$ and $z$ are distinct elements of $M$, then either
    $D_P(y)\subseteq D_P(z)$ or $D_P(z)\subseteq D_P(y)$.
  \end{claim*}

  In particular, the elements of $M$ can be labeled such that $M = \{y_1,\dots,y_m\}$ and $D_P(y_1) \subset \dots \subset D_P(y_m)$.
  Since $P$ is a component, $D_P(y_1) \neq \emptyset$.
  Since $P$ is a block, $D_P(y_1) \subsetneq D_P(y_m)$.

  \begin{claim*}
    $|M|=2$.
  \end{claim*}

  \begin{proofclaim}
    Suppose to the contrary that $|M|\ge3$.
    We form an inclusion representation $\cgS=\inrep{S_x:x\in P}$ 
    of $P$ by setting $S_{y_1}=\{y_2,y_m\}\cup D_P(y_1)$ and $S_x = D_P[x]$ for every $x \in P - \{y_1\}$.
    Since elements of $M$ are maximal in $P$, $\cgS$ is indeed an inclusion representation of $P$.
    Additionally, by~\Cref{pro:cw-ch}, $|S_x| \leq \ch(P)$ for every $x \in P - \{y_1\}$ and
        \[|S_{y_1}| = 2 + |D_P(y_1)| \leq 2 + (|D_P(y_m)-1) = |D_P[y_m]| \leq \ch(P).\]
    Furthermore, $|\bigcup \cgS| < |P|$ as $y_1 \notin \bigcup \cgS$, hence, $P$ is not a counterexample to the statement of the lemma.
    This contradiction yields that $|M|=2$.
  \end{proofclaim}

  Recall that now $M = \{y_1,y_2\}$ and $\emptyset \neq D_P(y_1) \subseteq D_P(y_2)$.
  For all $u,v \in P$, let $I(u,v)$ be the set of all elements $x$ in $P$ such that $u \leq x \leq v$ in $P$.
  Let $y \in P$ be the maximal such that $y$ covers at least two elements in $P$; $y \leq y_2$ in $P$; and $I(y,y_2)$ is a chain in $P$.
  A crucial property of $y$ is that for every $x \in P - I(y,y_2)$, we have $x < y_2$ in $P$ if and only if $x \leq y$ in $P$.
  Before we continue, we need to define one more element.
  By the Two Down Property, there exists $z \in P$ such that $D_P(y) \subset D_P(z)$ and $y \parallel z$ in $P$.
  In particular, $z \parallel y_2$ in $P$.
  Since $M = \{y_1,y_2\}$, it follows that $z \leq y_1$ in $P$.
  Moreover, since $D_P(y_1) \subset D_P(y_2)$, we obtain $z = y_1$.

  \begin{claim*}
  $y < y_2$ in $P$.
  \end{claim*}
  \begin{proofclaim}
      Suppose to the contrary that $y = y_2$.
      This yields $D_P(y_2) = D_P(y) \subseteq D_P(z) \subseteq D_P(y_1)$, which we know to be false.
  \end{proofclaim}

    Recall that by the definition of $z = y_1$, we have $D_P(y) \subset D_P(y_1)$.
    On the other hand, $D_P(y_1) \subset D_P(y_2)$ and $y \parallel y_1$, hence, by the definition of $y$, $D_P(y_1) \subset D_P(y)$.
    In particular, $D_P(y_1) = D_P(y)$.
    However, this shows that $P$ is not a block, which is a contradiction that ends the proof.
\end{proof}

\medskip
\section{Revisiting the Structure of Posets in \texorpdfstring{$\MIIR$}{MIIR}}
\label{sec:conclusion}
\medskip

In~\Cref{sec:dimcw}, we defined the classes of posets $\MTD$ and $\MCW$.
Theorems~\ref{thm:char:dim2} and~\ref{thm:char:cw} show that for a poset $P$, $\dim_2(P) = |P|$ if and only if $P \in \MTD$, and $\cw(P) = |P|$ if and only if $P \in \MCW$.
Also, recall that $\MIIR$ is the class of posets with $\iir(P) = |P|$, and we describe it in terms of three properties in~\Cref{thm:char:iir}.
In this section, we want to give a few remarks on the structure of posets in the mentioned classes.

The final steps in the proof of Theorem~\ref{thm:dim2-cw-block} show
that there are no blocks in either $\MTD$ or $\MCW$ that are also
components.  Moreover,
for a given integer $n$ with $n\ge5$, there is a unique block of cardinality $n$ in $\MTD$, which
is a poset in $\bbB$.  Also, there are two blocks of cardinality $n$
in $\MCW$, a poset from $\bbB$ and an $n$-element antichain.  The situation
with the class $\MIIR$ is far more complex.  For a large integer $n$, there
are exponentially many distinct blocks in $\MIIR$ that are components.  The 
next example explains how these blocks can be constructed.

\begin{example}
  Let $m$ and $n$ be integers, each of which is at least~$3$.
  $\sigma=(a_1,a_2,\dots,a_m)$ be a non-decreasing sequence
  of positive integers such that (1)~$a_1<n$; and (2)~$a_{m-1}=a_m=n$.  
  We associate with the sequence $\sigma$ a poset
  $P=P(\sigma)$ of height~$2$ defined as follows:
  \begin{itemize}[label = $\circ$]
    \item $P$ has $n$ minimal elements labeled $\{x_1,\dots,x_n\}$.
    \item $P$ has $m$ maximal elements labeled $\{y_1,\dots,y_m\}$.
    \item $x_i< y_j$ in $P$ if and only if $i\le a_j$.
  \end{itemize}
  It is an easy exercise to show that $P$ is a block, $P$ is a component,
  and $\iir(P)=|P|$.
  We show in Figure~\ref{fig:iir-block-component} the poset $P$ associated with
  the sequence $(1,1,1,2,2,2,2,4,4,4,8,8)$.
\end{example}

\begin{figure}[tp]
  \begin{center}
    \includegraphics{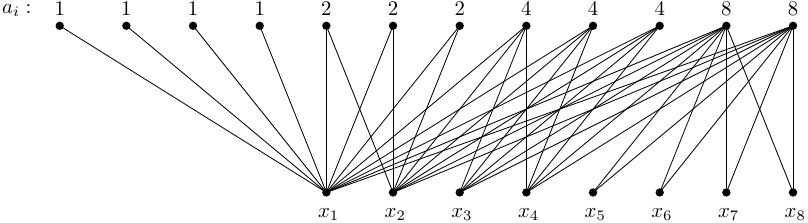}

  \end{center}
  \caption{A poset in $\MIIR$ that is both a block and a component.}
  \label{fig:iir-block-component}
\end{figure}

It is relatively straightforward to verify that all height~$2$ posets in $\MIIR$ that
are both blocks and components arise from the construction in the preceding
example.  However, we show in Figure~\ref{fig:cw-fig3} a poset $P$ of height~$6$
such that $P$ is both a block and a component.  It can be checked that
$P$ belongs to $\MIIR$. 

\begin{figure}[tp]
  \begin{center}
    \includegraphics{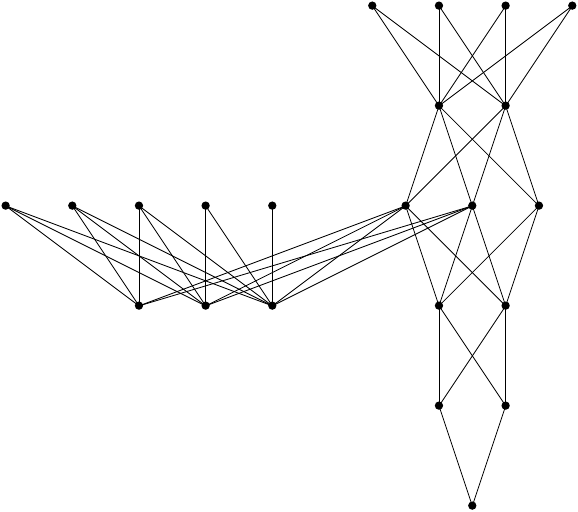}

  \end{center}
  \caption{A height~$6$ poset in $\MIIR$ that is a block and a component.}
  \label{fig:cw-fig3}
\end{figure}

In spite of these examples, it still makes sense to ask whether any
additional structural information can be gathered about properties
of blocks in $\MIIR$ that are also components.  We suspect that there
is a positive answer to this question.

\bibliography{bib}
\bibliographystyle{abbrv}
\end{document}